\documentclass{gtmon_a}
\pdfoutput=1

\usepackage[T1,T5]{fontenc}

\proceedingstitle{Proceedings of the School and Conference in Algebraic
Topology (The Vietnam National University, Hanoi, 9-20 August 2004)}
\conferencestart{9 August 2004}
\conferenceend{20 August 2004}
\conferencename{School and Conference in Algebraic Topology}
\conferencelocation{Vietnam National University, Hanoi, Vietnam}

\editor{John Hubbuck}
\givenname{John}
\surname{Hubbuck}

\editor{Nguy\~\ecircumflex{}n H V H\uhorn{}ng}
\givenname{H\uhorn{}ng}
\surname{Nguy\~\ecircumflex{}n}

\editor{Lionel Schwartz}
\givenname{Lionel}
\surname{Schwartz}

\title[The algebraic fifth transfer]{On behavior of the fifth algebraic
transfer}
\author{V\~{o} T\,N Qu\`ynh}
\surname{V\~{o}}
\givenname{Qu\`ynh}
\address{Department of Mathematics\\
Vietnam National University\\\newline
334 Nguyen Trai Street\\
Hanoi\\
Vietnam}
\email{quynhvtn@vnu.edu.vn}
\urladdr{}

\volumenumber{11}
\issuenumber{}
\publicationyear{2007}
\papernumber{14}
\startpage{309}
\endpage{326}

\doi{}
\MR{}
\Zbl{}

\arxivreference{}



\received{11 March 2005}
\revised{28 August 2005}
\accepted{8 September 2005}
\published{14 November 2007}
\publishedonline{14 November 2007}
\proposed{}
\seconded{}
\corresponding{}
\editor{}
\version{}

\makeatletter
\def\cnewtheorem#1[#2]#3{\newtheorem{#1}{#3}[section]
\expandafter\let\csname c@#1\endcsname\c@Theorem}

\newcommand{\tr}{\mathrm{tr}}
\newcommand{\Tr}{\mathrm{Tr}}
\newcommand{\Sq}{\mathrm{Sq}}
\newcommand{\Ext}{\mathrm{Ext}}

\newtheorem{Theorem}{Theorem}[section]
\cnewtheorem{Proposition}[Theorem]{Proposition}
\cnewtheorem{Lemma}[Theorem]{Lemma}
\cnewtheorem{Conjecture}[Theorem]{Conjecture}

\makeatother  

\makeautorefname{Lemma}{Lemma}
\makeautorefname{Theorem}{Theorem}
\makeautorefname{Proposition}{Proposition}

\newcommand{\F}{{\mathbb F}}
\newcommand{\Fd}{\F_2}
\newcommand{\cala}{{\cal A}}
\newcommand{\calL}{{\cal L}}
\newcommand{\V}{{\mathbb V}}  
\newcommand{\glk}{GL_k}

\newcommand{\glfive}{GL_5}

\newcommand{\otiglk}{{\otimes_{\glk}}}
\newcommand{\otiglfive}{{\otimes_{\glfive}}}
\newcommand{\oticala}{{\otimes_{\cala}}}

\begin{document}

\begin{htmlabstract}
In this paper, we show that Singer's fifth transfer is not an epimorphism
in degree 11.  More precisely, it does not detect the element P(h<sub>2</sub>)
in Ext<sub>A</sub><sup>5,16</sup>(<b>F</b><sub>2</sub>,<b>F</b><sub>2</sub>).
\end{htmlabstract}

\begin{abstract} 
In this paper, we show that Singer's fifth transfer is not an epimorphism
in degree 11.  More precisely, it does not detect the element $P(h_2)
\in \smash{\mathrm{Ext}_{\mathcal{A}}^{5, 16}(\mathbb{F}_2,\mathbb{F}_2)}$.
\end{abstract}

\maketitle

\section{Introduction and statement of results}\label{sec1}

Throughout the paper, the homology is taken with coefficients in $\Fd$.
Let $\V_k$ denote a $k$--dimensional $\Fd$--vector space, and
$PH_*(B\V_k)$ the primitive subspace consisting of all elements in
$H_*(B\V_k)$, which are annihilated by every positive-degree
operation in the mod 2 Steenrod algebra, $\cala$.  The
general linear group $GL_k:=GL(\V_k)$ acts regularly on $\V_k$ and
therefore on the homology and cohomology of $B\V_k$. Since the two
actions of $\cala$ and $GL_k$ upon $H^*(B\V_k)$ commute with each
other, there are inherited actions of $GL_k$ on $\Fd\oticala
H^*(B\V_k)$ and $PH_*(B\V_k)$. In \cite{Singer}, W Singer defined
the algebraic transfer
$$
\Tr_k \co \Fd{\otiglk}PH_d(B\V_k)\to \Ext_{\cala}^{k,k+d}(\Fd,\Fd)
$$
as an algebraic version of the geometrical transfer 
$\tr_k \co \pi_*^S((B\V_k)_+) \to \pi_*^S(S^0)$ to the stable homotopy groups
of spheres.

It has been proved that $\Tr_k$ is an isomorphism for $k=1,2$ by
Singer \cite{Singer} and for $k=3$ by Boardman \cite{Boardman}.
Among other things, these data together with the fact that
$\Tr=\bigoplus_{k} \Tr_k$ is an algebra homomorphism 
\cite{Singer} show that $\Tr_k$ is highly nontrivial. 

Therefore, the algebraic transfer is expected to be a useful tool
in the study of the mysterious cohomology of the Steenrod algebra,
$\Ext^{*,*}_{\cala}(\Fd, \Fd)$. In \cite{Hung97}, H\uhorn{}ng established
an attractive relationship between the algebraic transfer, the classical
conjecture on spherical classes, and the so-called ``hit'' problem.

Further, in \cite{Singer}, Singer gave computations to show that $\Tr_4$ is 
an isomorphism  in a range of degrees and recognized that $\Tr_5$ is not an 
epimorphism in degree 9. Then, he set up the following conjecture. 

\begin{Conjecture}[Singer \cite{Singer}]
\label{Singer}
$\Tr_k$ is a monomorphism for every $k$.
\end{Conjecture}

Recently, Bruner--Ha--H\uhorn{}ng showed in \cite{BHH} that $\Tr_4$  
does not detect the family $\{g_i\,| \, i\geq 0\}$. 
Furthermore, H\uhorn{}ng proved in \cite{Hung} that for every $k\geq 5$, 
there are infinitely many degrees in which  $\Tr_k$ is not an isomorphism.
Remarkably,  it has not been known whether the algebraic transfer fails to
be a monomorphism or fails to be an epimorphism for $k>5$.
Therefore, Singer's conjecture is still open.

The aim of this paper is to investigate the behavior of  $\Tr_5$ in degree 
11. 
We prove the following theorem.

\begin{Theorem}\label{main}
The element $P(h_2) \in \Ext_{\cala}^{5, 16}(\Fd,\Fd)$ is not in the 
image of the 
algebraic transfer $\Tr_5$. 
\end{Theorem}

Let $P_k:=H^*(B\V_k)$ be the polynomial algebra of $k$ variables, each of 
degree $1$. 
Then, the domain of $\Tr_k$, $\Fd\otiglk PH_*(B\V_k)$, is dual to 
$(\Fd\oticala P_k)^{GL_k}$.
In order to prove \fullref{main}, it suffices to show the following.

\begin{Proposition}\label{invariant-intro}
$(\Fd\oticala P_5)^{GL_5}_{11}=0$.
\end{Proposition}

Although our result does not give an answer to Singer's conjecture, 
it gives one more degree where the fifth algebraic  transfer fails to be 
an epimorphism. 

It should be noted that, R Bruner generously informed us that by using 
computer,  
he showed that $(\Fd\oticala P_5)_{11}$ is a 315--dimensional $\Fd$--vector 
space, 
and that its $GL_5$--invariant is zero. In this paper, we prove the 
proposition by using 
some convenient generators for $(\Fd\oticala P_5)_{11}$, which do not form 
a basis of the vector space. 

The paper is divided into four sections. \fullref{sec2} 
deals with the computation of minimal 
$\cala$--generators for the polynomial algebra $P_5$ in degree 11. Then, 
we prove \fullref{invariant-intro} and \fullref{main} in 
\fullref{sec3} and \fullref{sec4} respectively. 

\subsection*{Acknowledgments} 
I would like to express my warmest thanks to my supervisor, Professor 
Nguyen H\,V H\uhorn{}ng, 
for his inspiring guidance and  suggestions. 
I  am indebted to Professor R Bruner for kindly letting me use his 
computer's calculation on $(\Fd\oticala P_5)^{GL_5}_{11}$.
I am grateful to Dr Tran Ngoc Nam for many helpful discussions.
The work was supported in part by the National Research 
Program, Grant number 140 804.

\section[Computation of the indecomposables of P_5 in degree 11]
{Computation of the indecomposables of $P_5$ in degree 11.}
\label{sec2}

From now on, let us write $x=x_1$, $y=x_2$, $z=x_3$, $t=x_4$, $u=x_5$ and
denote the monomial $x^ay^bz^ct^du^e$ by $(a, b, c, d, e)$ for 
abbreviation.

\begin{Lemma}\label{generator}
The $\Fd$--vector space $(\Fd \oticala P_5)_{11}$ is generated by  
(the classes represented by) 
the following monomials and their permutations:
$$ 
\begin{matrix}    
(7, 3, 1, 0, 0), &  (7, 2, 1, 1, 0), &  (7, 1, 1, 1, 1), &  (5,3, 1, 1, 
1) \\
(5, 3, 3, 0, 0), &  (5, 3, 2, 1, 0), &  (3, 3, 3, 1, 1), &  (4, 3, 2, 1, 
1).                                                                    
\end{matrix}  
$$
\end{Lemma}

\begin{proof}
The monomials in the third column  are {\it spikes} in the meaning of 
W\,M Singer
\cite{Singer91} that  their exponents are all of the form $2^n-1$  for 
some $n$. 
It is well known that spikes do not appear in the expression of $\Sq^i Y$
for any $i$ positive and any monomial $Y$, since the powers
$x^{2^{n}-1}$ are not hit in the one variable case. 
Note that the elements in the first and second columns are respectively 
monomials which  depend only on three and four variables. The last two 
column's monomials depend on exactly five variables. 

Consider the projection $P_5\rightarrow \Fd\oticala P_5$. We show that under 
this projection, all monomials in degree 11 not listed in the Lemma go to 
zero except for the following six and permutations
\begin{align*}
(6,3,1,1) & \mapsto (5,3,2,1) + (5,3,1,2)\\
(4,3,3,1) & \mapsto (2,3,5,1) +(2,5,3,1)\\
(3,3,3,2) & \mapsto (2,3,5,1) +(2,5,3,1)+(3,2,5,1)+(5,2,3,1)+(3,5,2,1)\\
& \quad +  (5,3,2,1)\\
(5,2,2,1,1) & \mapsto (3,4,2,1,1)+(3,2,4,1,1)\\
(6,2,1,1,1) & \mapsto (3,4,2,1,1)+(3,2,4,1,1)+(3,4,1,2,1)+(3,2,1,4,1) \\
&  \quad +   (3,4,1,1,2)+(3,2,1,1,4)\\
(3,3,2,2,1) & \mapsto (5,3,1,1,1)+(3,5,1,1,1)+(4,3,1,1,2)+(3,4,1,1,2).
\end{align*} 
As the action of the Steenrod algebra on $P_5$ commutes with that of the 
general linear group $GL_5$, without loss of generality, we need only to 
consider  monomials $(a, b, c, d, e)$ in degree 11 of $P_5$ with $a \geq b 
\geq c\geq d\geq e$. We have  the  following five cases. 

\textbf{Case 1}\qua The monomial $(a, b, c, d, e)$ depends only on
one variable, $(a, b, c, d, e)=(a, 0, 0, 0, 0)$ with $a\neq 0$.  There is
only one such a monomial in degree 11  of $P_5$, namely $(11,0,0,0,0)$ .
It is hit because 
$$
(11,0,0,0,0) = \Sq^4 (7,0,0,0,0). 
$$
\textbf{Case 2}\qua The monomial $(a, b, c, d, e)$ depends on exactly two 
variables, 
$(a, b, c, d, e)=(a, b, 0, 0, 0)$, where $a$ and $b$ are nonzero. It is 
also hit, as we have
\begin{eqnarray*}
(10, 1, 0,0,0) & = &\Sq^4 (6,1,0,0,0)\\
(9, 2, 0,0,0) & = &\Sq^4 (5,2,0,0,0)\\
(8, 3, 0,0,0) & = &\Sq^4 (4,3,0,0,0)\\
(7, 4, 0,0,0) & = &\Sq^4 (5,2,0,0,0)+ \Sq^2(7,2,0,0,0)\\
(6, 5, 0,0,0) & = &\Sq^4 (4,3,0,0,0) + \Sq^2(6,3,0,0,0).
\end{eqnarray*}
\textbf{Case 3}\qua The monomial $(a, b, c, d, e)$ depends exactly on three 
variables
$(a, b, c, d, e)=(a, b, c, 0, 0)$, where $a$, $b$ and $c$ are nonzero.
This should be one of the following monomials: 
\begin{eqnarray*}
&& (7,3,1,0,0), (5,3,3,0,0) \\
&& (9,1,1,0,0), (8,2,1,0,0), (7,2,2,0,0), (6,4,1,0,0),(6,3,2,0,0), 
(5,4,2,0,0) \\
&& (5,5,1,0,0), (4,4,3,0,0). 
\end{eqnarray*} 
The first two monomials are listed in the lemma. 
The last eight monomials  are killed by the Steenrod algebra, since we have
\begin{eqnarray*}
(9,1,1,0,0)&=& \Sq^4(5,1,1,0,0)\\
(8,2,1,0,0)&=& \Sq^4(4,2,1,0,0)\\
(7,2,2,0,0)&=& \Sq^1(7,2,1,0,0)+ \Sq^4(4,2,1,0,0)\\
(6,4,1,0,0)&=& \Sq^2(6,2,1,0,0)+ \Sq^4(4,2,1,0,0)\\
(6,3,2,0,0)&=& \Sq^1(6,3,1,0,0)+ \Sq^2(6,2,1,0,0)+ \Sq^4(4,2,1,0,0)\\
(5,4,2,0,0)&=& \Sq^1(5,4,1,0,0)+ \Sq^2(6,2,1,0,0)+ \Sq^4(4,2,1,0,0), 
\end{eqnarray*}
and 
\begin{eqnarray*}
(5,5,1,0,0)& =& (6,4,1,0,0)+(6,3,2,0,0)+(5,4,2,0,0)+\Sq^2(5,3,1,0,0)\\
(4,4,3,0,0)&=&(4,2,5,0,0) +\Sq^2(4,2,3,0,0).
\end{eqnarray*}
\textbf{Case 4}\qua  The monomial $(a, b, c, d, e)$ depends exactly on four 
variables, 
$(a, b, c, d, e)=(a, b, c, d, 0)$, where $a$, $b$, $c$  and $d$ are non 
zero.
This should be one of the following monomials: 
\begin{align*}
&(7, 2, 1, 1, 0),  (5, 3, 2, 1, 0)\\
&(8,1,1,1,0) , (6,2,2,1,0), (5,2,2,2,0), (4,4,2,1,0), (4,3,2,2,0), 
(5,4,1,1,0)  \\
&(6,3,1,1,0), (4,3,3,1,0), (3,3,3,2,0).  
\end{align*} 
The first two monomials are listed in the lemma. 
The next six  monomials  are killed by the Steenrod algebra, since we have
\begin{eqnarray*}
(8,1,1,1,0) & = & \Sq^4(4,1,1,1,0)\\
(6,2,2,1,0)&=&\Sq^5(3,1,1,1,0) + \Sq^4(4,1,1,1,0)\\
(5,2,2,2,0)&= & (6,2,2,1,0) +  \Sq^1(5,2,2,1,0)\\
(4,4,2,1,0)&=&\Sq^4(2,2,2,1,0)+\Sq^2(2,2,4,1,0)\\
(4,3,2,2,0)&=& (4,4,2,1,0) +\Sq^1(4,3,2,1,0)\\
(5,4,1,1,0)&=&(4,4,2,1,0)+(4,4,1,2,0)+(3,4,2,2,0)+ \Sq^2(3,4,1,1,0).
\end{eqnarray*} 
The last three monomials $(6,3,1,1,0), (4,3,3,1,0), (3,3,3,2,0)$ can be  
expressed in terms of the monomials $(7, 2, 1,1 ,0), (5, 3, 2, 1, 0)$  and 
their permutations. Indeed, we get the following equalities
\begin{eqnarray*}
(6,3,1,1,0) & = & (5,3,2,1,0)+(5,3,1,2,0)+(5,4,1,1,0)+\Sq^1(5,3,1,1,0)\\
(4,3,3,1,0) &= & (2,3,5,1,0) +(2,5,3,1,0) +(2,4,4,1,0) +(2,3,4,2,0) \\
&   & +(2,4,3,2,0) + \Sq^2(2,3,3,1,0) \\           
(3,3,3,2,0) &= & (4,3,3,1,0) +(3,4,3,1,0)+(3,3,4,1,0) + \Sq^1(3,3,3,1,0).
\end{eqnarray*}
\textbf{Case 5}\qua The monomial $(a, b, c, d, e)$ depends exactly on five 
variables, 
$(a, b, c, d, e)=(a, b, c, d, e)$, where $a$, $b$, $c$, $d$  and $e$ are 
nonzero.
This should be one of the following monomials: 
\begin{eqnarray*}
&&  (7, 1, 1, 1, 1), (5, 3, 1, 1, 1), (3, 3, 3, 1, 1),  (4, 3, 2, 1, 1)  
\\
&&  (4, 4, 1, 1, 1), (4, 2, 2, 2, 1), (3, 2, 2, 2, 2) \\
&&  (5, 2, 2, 1, 1), (6, 2, 1, 1, 1), (3, 3, 2, 2, 1).  
\end{eqnarray*}
The first four monomials are listed in the lemma. 
The next three  monomials  are hit by the Steenrod algebra, since we have
\begin{eqnarray*}
(4,4,1,1,1)&=&\Sq^2(4,2,1,1,1)+\Sq^2(2,4,1,1,1)+\Sq^4(2,2,1,1,1)\\
(4,2,2,2,1)&=&\Sq^4(2,2,1,1,1)+ \Sq^2(2,4,1,1,1)+\Sq^1(4,2,1,1,2)\\
(3,2,2,2,2) &= &(4,2,2,2,1) + \Sq^1(3,2,2,2,1).
\end{eqnarray*}
The last three monomials $(5,2,2,1,1), (6,2,1,1,1), (3,3,2,2,1) $ are 
expressed in terms 
of the monomials $(5,3,1,1,1), (4,3,2,1,1)$ and their permutations. Indeed
\begin{eqnarray*}
(5,2,2,1,1) &=& (3,4,2,1,1)+(3,2,4,1,1)+(4,2,2,2,1) +(4,2,2,1,2) \\
&  &+(3, 2, 2, 2, 2) + \Sq^2(3,2,2,1,1) \\                   
(6,2,1,1,1) &= & (5,2,2,1,1)+(5,2,1,2,1)+(5,2,1,1,2) +\Sq^1(5, 2, 1,1,1)\\
(3,3,2,2,1) &=& (5,3,1,1,1)+(3,5,1,1,1)+(4,4,1,1,1)+(4,3,1,1,2) \\
& & +(3,4,1,1,2) +\Sq^2(3,3,1,1,1) \\     
& &+ \Sq^1(3,3,1,1,2)+\Sq^1(4,3,1,1,1)+\Sq^1(3,4,1,1,1).
\end{eqnarray*}
The lemma is proved.
\end{proof} 

We denote by $A, B, C, D, E, F, G, H$   the families of all permutations of 
the following monomials respectively 
\begin{eqnarray*}
&& (7,3,1,0,0), (5,3,3,0,0), (7,2,1,1,0), (5,3,2,1,0) \\
&& (7,1,1,1,1), (3,3,3,1,1), (5,3,1,1,1), (4,3,2,1,1).
\end{eqnarray*}
For $X$ one of the families $A,  B, C, D, E, F, G, H$, let $\calL (X)$ be 
the vector subspace of $(\Fd \oticala P_5)_{11}$ spanned by all the 
elements of the family $X$. Further, set 
$\calL (G,H)= \calL (G) + \calL (H)$.

\begin{Lemma}\label{linearrelation}
Every  $p \in (\Fd\oticala P_5)_{11}$ can be expressed uniquely as a sum
$$ p = p_{A}+ p_{B}+p_C+p_D+p_E+p_F+p_{(G,H)},$$
where $p_X  \in  \calL(X)$ for $X \in  \{A, B, C, D, E, F \}$  and 
$p_{(G,H)}\in \calL (G,H).$ 
\end{Lemma}

\begin{proof}
By \fullref{generator}, 
if  $p \in (\Fd\oticala P_5)_{11}$  then $p$ can be 
expressed  as a sum of elements in $\calL (A), \calL (B), \calL (C), \calL 
(D), \calL (E), \calL (F)$ and  in $\calL (G, H)$. 
In order to prove the uniqueness of the expression we now 
suppose that there is a linear relation 
$$ p_{A}+ p_{B}+p_C+p_D+p_E+p_F+p_{(G,H)}= 0 $$
in $(\Fd\oticala P_5)_{11}$, where $p_X  \in  \calL(X) $ for 
$X \in  \{A, B, C, D, E, F \}$  and $X= (G,H).$ 
We need to show 
$p_{A}= p_{B}=p_C=p_D=p_E=p_F=p_{(G,H)}= 0$ in $(\Fd\oticala P_5)_{11}.$
First, we note that $p_A = p_E = p_F = 0$, 
as  $p_A, p_E, p_F$  are expressed in terms of the spikes,  which do not 
appear in the expression of $\Sq^i(Y)$ for any $i$ positive and any 
monomial $Y$.
Hence
$$  p_{B}+p_C+p_D+p_{(G,H)}= 0. $$
Consider the homomorphism $\pi_{tu} \co \Fd\oticala P_5\rightarrow 
\Fd\oticala P_3$ induced by the projection $P_5\rightarrow P_5/(t, u)\cong 
P_3.$ Under this homomorphism, the image of the above linear relation is 
$\pi_{tu} (p_B)= 0.$ 
Using all the projections from $P_5$ to its quotients by the ideals 
generated 
by any pairs of the five variables $x, y, z, t, u$, we get $p_B = 0.$
Hence
$$  p_C+p_D+p_{(G,H)}= 0. $$
Next, we
consider the homomorphism 
$\pi_{u} \co \Fd\oticala P_5\rightarrow \Fd\oticala P_4$ 
induced by the projection $P_5\rightarrow P_5/( u)\cong P_4.$ 
Let $\pi_{u}$ act on  both sides of the above equality, we get 
$$\pi_{u}(p_C) + \pi_{u}(p_D) = 0,$$
where $\pi_{u}(p_C)$  is a linear combination of permutations of element 
$(7,2,1,1).$
As $7 $ and $1$ are  of the form $2^n-1$, the monomial $(7,2,1,1)$  
appears  only 
as  a term  in $\Sq^i(a, b, c, d)$ for $i=1$ and $(a, b, c, d)= (7,1,1,1)$ 
as follows
$$\Sq^1(7,1,1,1)= (8,1,1,1)+ (7,2,1,1)+ (7,1,2,1)+ (7,1,1,2) .$$
So, $\pi_u(p_C)$ contains $(7,2,1,1)$ as a term if and only if it also 
contains $(7,1,2,1)+ (7,1,1,2) $. 
A consequence of the above expression of $\Sq^1(7,1,1,1)$ is
$$(7,2,1,1)+ (7,1,2,1)+ (7,1,1,2)= 0, $$
since  $(8,1,1,1)= \Sq^4(4,1,1,1)$. 
Thus, $\pi_u(p_C) =0$, and therefore  $\pi_u(p_D) = 0.$ 

In the above argument, replacing the homomorphism $\pi_{u}$ by any of $\pi_{x},  \pi_{y}, \pi_{z},  \pi_{t}$, and we get
$$p_C = p_D = p_{(G, H)}= 0.
\proved
$$
\end{proof}

The following Lemma is a consequence of  \fullref{generator}  and 
\fullref{linearrelation}.  
\begin{Lemma}\label{decomposition}
There is  a decomposition 
of $F_2$--vector spaces
$$ (\Fd \oticala P_5)_{11}= \calL(A) \oplus \calL(B)\oplus \calL(C)\oplus 
\calL(D)\oplus \calL(E)\oplus \calL(F) \oplus \calL(G, H).$$
\end{Lemma}

\section[GL_5 invariants of the indecomposables of P_5]
{$GL_5$--invariants of the indecomposables of $P_5$ in degree $11$}
\label{sec3}

The goal of this section is to prove the following proposition, which is 
also numbered 
as \fullref{invariant-intro} in the introduction.

\begin{Proposition}\label{invariant}
$(\Fd \oticala P_5)_{11}^{GL_5}= 0.$
\end{Proposition}

Let $S_5$ be the symmetric group  on $5$ letters $x, y, z, t, u$. 
It is easy to see that $\calL(A)$, $\calL(B)$, $\calL(C)$, $\calL(D)$,
$\calL(E)$, 
$\calL(F)$ and $\calL(G, H)$ are all $S_5$--submodules.  So the equality in 
\fullref{decomposition}
$$ 
(\Fd \oticala P_5)_{11} = \calL(A) \oplus \calL(B)\oplus \calL(C)\oplus 
\calL(D)\oplus \calL(E)\oplus \calL(F) \oplus \calL(G, H)
$$
is  a decomposition of $S_5$--modules. 

By \fullref{linearrelation},  
every  $p \in (\Fd\oticala P_5)_{11}$ can be expressed uniquely as a sum
$$ p = p_{A}+ p_{B}+p_C+p_D+p_E+p_F+p_{(G,H)},$$
where $p_X  \in  \calL(X)$ for $X \in  \{A, B, C, D, E, F \}$  and 
$p_{(G,H)}\in \calL (G,H).$  So, each term of the sum  is an 
$S_5$--invariant.

For $X$ one of the letters $A, B, C, D, E, F$, let $x_i$ be the 
coefficient in the above expression of $p$ of the $i$th monomial in the 
family $X$ ordered lexicographically. 
Note that all monomials of families $A, E$ and $F$ are spikes. 
It is well known that spikes do not appear in the expression of $\Sq^i Y$
for any $i$ positive and any monomial $Y$. Hence, the coefficient of any 
spike
is zero in every linear relation in $\Fd\oticala P_5$.
It implies that, in the expression of $p$, the coefficients of monomials 
in each 
of the families $A, E, F$ are equal to each other. 

\fullref{invariant} is proved by combining the following five lemmas.

\begin{Lemma}\label{AB}
If  $p = p_{A}+ p_{B}+p_C+p_D+p_E+p_F+p_{(G,H)}$ 
is the decomposition of $p \in  \smash{(\Fd \oticala P_5)_{11}^{GL_5}}$
as in  \fullref{linearrelation}, then $p_A = p_B = 0.$
\end{Lemma}

\begin{proof}
With $\pi_{tu}$ defined as in the proof of \fullref{linearrelation}, we 
have
$$ 
\pi_{tu}(p) = \pi_{tu}(p_A)+ \pi_{tu}(p_B).
$$
We have
$$\pi_{tu}(p_B) = b_1(5,3,3)+ b_2(3,5,3)+b_3(3,3,5). $$

According to the argument given above,  the coefficients $a_i$ are equal 
each other.   
Set $a= a_i$ and we have
$$
\pi_{tu}(p_A) = a[(7,3,1)+ (7,1,3)+(3,7,1)+ (3,1,7)+ (1,7,3)+ (1,3,7)].
$$
We will show that $b_1= b_2 =b_3$ and $a=0.$
Associated to the two variables $x$ and $y$,  let $\sigma_{xy}$ be the 
transposition of $x$ and $y$ 
that keeps the other variables fixed.

As $p$ is a $GL_5$--invariant in  $\Fd \oticala P_5$, we have
$$\pi_{tu}(\sigma_{xy}(p)+p) = \pi_{tu} (0) = 0  
\text{ in }  \Fd \oticala P_3,$$
equivalently
$$\sigma_{xy}(\pi_{tu}(p) ) + \pi_{tu}(p )= 0
\text{ in }  \Fd \oticala P_3. $$

Combining $\pi_{tu}(p_A)=a[(7,3,1)+\mbox{\rm symmetrized}]$ with the fact 
that the monomial $(7,3,1)$ is  spike, we have 
$$\sigma_{xy}(\pi_{tu}(p_A) ) + \pi_{tu}(p_A )= 0.$$
From this, it follows that 
$$\sigma_{xy}(\pi_{tu}(p_B))  + \pi_{tu}(p_B )= 0,$$
or equivalently
$$(b_1 + b_2)((5,3,3)+ (3,5,3)) = 0.$$
However,  
\begin{eqnarray*}
(5,3,3)+(3,5,3)+(3,3,5)  &= & \Sq^2(3,3,3)+ \Sq^1(4,3,3)+\Sq^4(4,2,1)\\ 
&  &+ \Sq^2(6,2,1)+\Sq^1(5,4,1)+ \Sq^2(3,4,2).
\end{eqnarray*}
So, we get 
$$ (b_1+b_2)(3,3,5)=0.$$

On the other hand, the linear transformation $x\mapsto x+z, y\mapsto y, 
z\mapsto z $ 
sends $(3,3,5)$ to $(3,3,5)+ (2,3,6) +(1,3,7) +(0,3,8)\sim (3,3,5)+ 
(1,3,7).$ 
As the action of the Steenrod algebra commutes with linear maps, 
if $(3,3,5)$ is hit then so is $(1,3,7).$ This is impossible, because 
$(1,3,7)$ is a spike. 
Thus, $(3,3,5) \neq 0  $ in $ \Fd \oticala P_5 $ and therefore $b_1 + 
b_2 = 0,$ or  
$b_1 = b_2.$
By similarity, using all transpositions of any pairs of the three 
variables $x, y, z,$ 
we get $b_1= b_2= b_3.$
Hence
$$ \pi_{tu}(p_B )= b_1[(5,3,3)+(3,5,3)+(3,3,5)] = b_1. 0 = 0.$$

By the symmetry of the variables, we also obtain $\pi_{ij}(p_B) = 0$,
where $(i,j)$ is any pair of the five variables $x, y, z, t , u.$
Thus $p_B  = 0.$  

In order to prove $a=0,$ we consider  the linear transformation, 
$\omega_{xy }$,   
that sends $x$ to $x+y$ and keeps the other  variables fixed.
As $p_B = 0,$ we have
$
\pi_{tu}(p)= \pi_{tu}(p_A).
$
From  $\omega_{xy}(p)+p =0$, it follows that
$$
\omega_{xy}(\pi_{tu}(p_A)) + \pi_{tu}(p_A) = 0, 
$$
or equivalently
$$
a[(5,3,3)+ (3,5,3)+(1,7,3)+(3,7,1)+(1,3,7)] = 0.
$$
Combining this with the fact that $(1,7,3)$ is a spike, we get $a= 0.$
\end{proof}

\begin{Lemma}\label{CD}
If  $p = p_{A}+ p_{B}+p_C+p_D+p_E+p_F+p_{(G,H)}$ 
is the decomposition of $p \in  \smash{(\Fd \oticala P_5)_{11}^{GL_5}}$
as in  \fullref{linearrelation},  then $p_C = p_D = 0.$
\end{Lemma}

\begin{proof}
By \fullref{AB}, $p_A = p_B = 0.$  As a consequence,  $p= p_C +p_D+ p_E+ 
p_F + p_{(G, H)}.$
Let $\pi_u \co \Fd \oticala P_5 \to \Fd \oticala P_4$ be the homomorphism 
induced by the projection $P_5 \to P_5/( u) \cong P_4$ as in the proof of 
\fullref{linearrelation}. 
We have
$$ \pi_u (p) = \pi_u(p_C) + \pi_u(p_D),$$
where  $\pi_{(u)}(p_C)$ and   $\pi_{(u)}(p_D)$ are respectively certain 
linear combinations of permutations of the elements $(7,2,1,1)$ and 
$(5,3,2,1)$.

In the families $\pi_u(C), \pi_u(D)$, there are exactly three monomials 
$(x,y,z,t)$ with $t=7$, namely
$$ (2,1,1,7), (1,2,1,7), (1,1,2,7).$$
We have $\Sq^1(1,1,1,7) =  (2,1,1,7) +  (1,2,1,7)+ (1,1,2,7),$  and hence
$$
(2,1,1,7) =  (1,2,1,7)+ (1,1,2,7)  \text{ in } \Fd \oticala P_4.
$$
So we get
$$
\pi_u( p) = c_1(1,2,1,7)+c_2(1,1,2,7)+ \text{ terms of the form }
(x,y,z,t) \text{ with } t\neq 7. 
$$
Let $\omega_{xy}$ be the transposition of $x$ and $y$ as defined in the 
proof of \fullref{AB}. 
It is easily seen that 
\begin{align*}
\omega_{xy}(c_1(1,2,1,7)+c_2(1,1,2,7))& = c_1(1,2,1,7)+c_2(1,1,2,7)+ 
c_1(0,3,1,7) \\
& \qua + c_2(0,2,2,7).
\end{align*}
Combining this with the fact that $\omega_{xy}(\pi_u(p)) + \pi_u(p) =0$, 
we obtain
$c_1 = 0$, as $(0,3,1,7)$ is a spike.  

By a similar argument using $\omega_{xz}$, we get $c_2= 0.$  Hence 
$\pi_u(p_C)= 0.$ 

By the symmetry of the variables, we have  
$$
\pi_x(p_C) = \pi_y(p_C)= \pi_z(p_C)=\pi_t(p_C)= \pi_u(p_C) = 0.
$$ 
As a consequence, we get  $p_C = 0.$

Similarly,  in order to prove  $p_D = 0 $ we need only to show that 
$\pi_u(p_D) = 0.$
The family $\pi_u(D)$, which consists of  all the permutations of the 
monomials $(5,3,2,1),$ 
has twenty-four  elements.
A direct calculation shows  the following table.

{\small
\begin{center}
\begin{tabular}{|l|l|l|l|} \hline
monomial & $\omega_{xy}(\text{monomial}){+}\text{monomial}$ &
monomial & $\omega_{xy}(\text{monomial}){+}\text{monomial}$ \\ \hline
(5,3,2,1)    & (1,7,2,1) &
(5,3,1,2)    & (1,7,1,2) \\
(5,2,3,1)    & (4,3,3,1)+(1,6,3,1)+(0,7,3,1) &
(5,2,1,3)    & (4,3,1,3)+(1,6,1,3)+(0,7,1,3) \\
(5,1,3,2)    & (1,5,3,2) &
(5,1,2,3)    & (1,5,2,3) \\
(3,5,2,1)    & (1,7,2,1) &
(3,5,1,2)    & (1,7,1,2) \\
(3,2,5,1)    & (2,3,5,1) &
(3,2,1,5)    & (2,3,1,5) \\
(3,1,5,2)    & (1,3,5,2) &
(3,1,2,5)    & (1,3,2,5) \\
(2,5,3,1)    & (0,7,3,1) &
(2,5,1,3)    & (0,7,1,3) \\
(2,3,5,1)    & 0 &
(2,3,1,5)    & 0 \\
(2,1,5,3)    & (0,3,5,3) &
(2,1,3,5)    & (0,3,3,5) \\
(1,5,3,2)    & 0 &
(1,5,2,3)    & 0 \\
(1,3,5,2)    & 0 &
(1,3,2,5)    & 0 \\
(1,2,5,3)    & (0,3,5,3) &
(1,2,3,5)    & (0,3,3,5). \\ \hline
\end{tabular}
\end{center}}

Let $d_{(a,b,c,d)}$ be the coefficient of the monomial $(a,b,c,d)$ in the 
expression of $\pi_u(p_D)$.  
Since $\pi_u(p_D)$ is a $GL_4$--invariant, we have
$\omega_{xy}(\pi_u(p_D)) + \pi_u(p_D) = 0  
\text{ in } \Fd \oticala P_4. $
Combining this and the above table we obtain
\begin{align*}
[d_{(5,3,2,1)}+ d_{(3,5,2,1)}](1,7,2,1) 
+[d_{(5,3,1,2)}+d_{(3,5,1,2)}](1,7,1,2) & = 0 \\
[d_{(5,2,3,1)}+ d_{(2,5,3,1)}](0,7,3,1) 
+[d_{(5,2,1,3)}+d_{(2,5,1,3)}](0,7,1,3) & = 0 \\
[d_{(2,1,5,3)}+ d_{(1,2,5,3)}](0,3,5,3) 
+[d_{(2,1,3,5)}+d_{(1,2,3,5)}](0,3,3,5) & = 0.  
\end{align*}
As $(0,7,3,1)$ and $(0,7,1,3)$ are spikes, we get
$$d_{(5,2,3,1)}= d_{(2,5,3,1)}
  \quad\text{and}\quad
  d_{(5,2,1,3)}= d_{(2,5,1,3)}.$$
Let $\omega_{xz}$ be the linear transformation which sends $x$ to $x+z$ 
and keeps the other variables fixed. 
Applying $\omega_{xz}$  to the above first  equality, we get
$$[d_{(5,3,2,1)}+ d_{(3,5,2,1)}](0,7,3,1) 
+[d_{(5,3,1,2)}+d_{(3,5,1,2)}](0,7,2,2) = 0.  $$
It implies
$d_{(5,3,2,1)} = d_{(3,5,2,1)}$ and similarly 
$d_{(5,3,1,2)} = d_{(3,5,1,2)}$.

Similarly, it follows from the third equality that                         
$$d_{(2,1,5,3)} = d_{(1,2,5,3)}
  \quad\text{and}\quad
  d_{(2,1,3,5)} = d_{(1,2,3,5)}.$$
It is easy to see that the symmetric group on the four 
letters $\{5,3,2,1 \}$ is generated by
the transpositions  $(5,3), (5,2), (2,1).$
Combining this with the above equalities, it implies that all coefficients 
$d_{(a,b,c,d)}$ are the same. 
Let us denote this common coefficient by $d$.  We have
\begin{eqnarray*}
\omega_{xy}(\pi_u(p_D)) +\pi_u(p_D)& = 
&d[(4,3,3,1)+(1,6,3,1)+(4,3,1,3)+(1,6,1,3)\\
&  & +(1,5,3,2) +(1,5,2,3)+(2,3,5,1)+(2,3,1,5) \\
&  & +(1,3,5,2)+(1,3,2,5)].
\end{eqnarray*}
As shown in the proof of \fullref{generator}, we get
\begin{align*}
(4,3,3,1)& = (2,3,5,1)+(2,5,3,1) \\
(1,6,3,1)& = (2,5,3,1)+(1,5,3,2) \\
(4,3,1,3)& = (2,3,1,5)+(2,5,1,3) \\
(1,6,1,3)& = (2,5,1,3)+(1,5,2,3).
\end{align*}
Hence, the above equality is reduced to
$$ d[(1,3,5,2)+(1,3,2,5)] = 0.$$
Applying $\omega_{xt}$ to this relation, we get
$d[(0,3,5,3)]= 0.$
It implies $d= 0$, since we have shown  that $(0,3,5,3)$ is nonzero.

So $\pi_u(p_D) = 0$ and therefore $p_D =0$.
\end{proof}

\begin{Lemma}\label{E}
If  $p = p_{A}+ p_{B}+p_C+p_D+p_E+p_F+p_{(G,H)}$ 
is the decomposition of $p \in  \smash{(\Fd \oticala P_5)_{11}^{GL_5}}$
as in  \fullref{linearrelation},  then $p_E = 0.$
\end{Lemma}

\begin{proof}
According to the above two lemmas, $p = p_ E +p_F+ p_{(G, H)}.$

As  $(7,1,1,1,1)$  is a spike, the coefficients of its all permutations in 
the expression of 
$p \in  \smash{(\Fd \oticala P_5)_{11}^{GL_5}}$ are equal to each other. 
We denote this common coefficient by $e$.

So, $p_E$ can be written in the form
$$p_E = e [(7,1,1,1,1)+(1,7,1,1,1)+(1,1,7,1,1)+(1,1,1,7,1)+(1,1,1,1,7)],$$
where $e \in \Fd.$ 

In the families $E$, $F$, $G$, $H$ there is exactly one monomial with
$u{=}7$, namely\break $(1,1,1,1,7)$.
Let $\sigma$ be the linear transformation that sends $x$ to $x+z$, $y $ to 
$y+z$ 
and keeps the other variables fixed.

An easy computation shows
$$\sigma (1,1,1,1,7) = (1,1,1,1,7) + (1,0,2,1,7)+(0,1,2,1,7)+ (0,0,3,1,7).$$ 
Note that the images under $\sigma$ of the other monomials of the families 
$E, F, G, H$ 
in the expression of $p$ do not contain the spike $(0,0,3,1,7)$. 

So, $\sigma(p)+p$ contains $e (0,0,3,1,7)$ as a term. 
It implies $\alpha = 0$,  and therefore $p_E = 0.$
\end{proof}

\begin{Lemma}\label{F}
If  $p = p_{A}+ p_{B}+ p_C+ p_D+ p_E+ p_F+p_{(G,H)}$ 
is the decomposition of $p \in  \smash{(\Fd \oticala P_5)_{11}^{GL_5}}$
as in  \fullref{linearrelation},  then $p_F = 0.$
\end{Lemma}

\begin{proof}
According to the above three lemmas, we have $p = p_F + p_{(G, H)}.$ 

By the same argument  given in the previous lemma, 
as $(3,3,3,1,1)$  is a spike, the coefficients of its all permutations in 
the expression of 
$p \in  (\Fd \oticala P_5)_{11}^{GL_5}$ are equal each other. 
We denote this common coefficient by $f$.

In the family $F, G, H$, there are exactly two monomials with $z=3, t=3, 
u=1$, namely
$$ (3,1,3,3,1), (1,3,3,3,1).$$
As $p$ is a $GL_5$--invariant in $\Fd \oticala P_5$, we have particularly
$$\omega_{xy}(p) + p = 0.$$
A routine computation shows
\begin{eqnarray*}
\omega_{xy}(3,1,3,3,1)+ (3,1,3,3,1) &=& (2,2,3,3,1) +(1,3,3,3,1)+ 
(0,4,3,3,1) \\
\omega_{xy}(1,3,3,3,1)+ (1,3,3,3,1) &=& (0,4,3,3,1).
\end{eqnarray*}
Note that the images under $\omega_{xy}$ of the other monomials of the 
families $F, G, H$ 
in the expression of $p$ do not contain the spike $(1,3,3,3,1).$

Thus, $\omega_{xy}(p) + p$ contains $f (1,3,3,3,1)$ as a term.  
This  implies $f = 0$   and therefore $p_F = 0.$ 
\end{proof}


\begin{Lemma}\label{G,H}
If  $p = p_{A}+ p_{B}+ p_C+ p_D+ p_E+ p_F+p_{(G,H)}$ 
is the decomposition of $p \in  \smash{(\Fd \oticala P_5)_{11}^{GL_5}}$
as in  \fullref{linearrelation},  then $p_{(G,H)} = 0.$
\end{Lemma}

\begin{proof}
According to the above four lemma, we have $p = p_{(G, H)}.$ 
Recall  that $p_{(G, H)}$ is expressed in terms 
of the elements of the families $G$ and $H.$ 

The proof is divided into 2 steps.

\textbf{Step 1}\qua Let $K$ be the family of all variable permutations of  
monomial  $(3,3,2,2,1).$
We will show that $p$ can be expressed in terms of the elements of  the 
family $K.$

The elements in  family $G$ are divided into pairs by twisting the 
variables 
whose exponents are 5 and 3. 

Consider  two monomials $(5,3,1,1,1), (3,5,1,1,1)$ in one of the pairs.
With  $\omega_{xy}$ as defined in the proof of \fullref{AB}, we have
$$\omega_{xy}(5,3,1,1,1) = (5,3,1,1,1) + (4,4,1,1,1)+ (1,7,1,1,1)+ 
(0,8,1,1,1) $$
$$\omega_{xy}(3,5,1,1,1) = (3,5,1,1,1) + (2,6,1,1,1)+ (1,7,1,1,1)+ 
(0,8,1,1,1). $$
Further, $(1,7,1,1,1)$ does not appear in 
the expressions of the images under  $\omega_{xy}$  of any other elements 
in $G, H$.
As $p$ is a $GL_5$--invariant, it satisfies
$$\omega_{xy}(p) + p = 0 \text{ in }W \Fd \oticala P_5.$$
However, $(1,7,1,1,1)$ is a spike, which does not appear in the expression 
of $\Sq^iY$ 
for any $i$ positive and any monomial $Y$.  
So, the coefficients of the monomials $(5,3,1,1,1)$ and $(3,5,1,1,1)$ 
in the expression of $p$ are equal each other.

On the other hand, by using by $\Sq^2(3,3,1,1,1) 
+\Sq^1(4,3,1,1,1)+\Sq^1(3,4,1,1,1)$ we get
$$(5,3,1,1,1)+ (3,5,1,1,1) =  (3,3,2,2,1)+(3,3,1,2,2)+(3,3,2,1,2).$$
Then, in the expression of $p$, the sum of monomials in family $G$ can be 
written as 
a sum of monomials in family $K.$

Next, we consider in the expression of $p$ the sum of monomials in the 
family $H$.  

First, we consider the set of monomials of the forms $(4,3,c,d,e)$ and 
$(3,4,c,d,e)$ 
in the family $H$.  Then,  $(c,d,e)$ is  a permutation of $(2,1,1).$  
We will show that the sum of  the monomials in this set occurring in the 
expression of $p$  
equals to the sum of some monomials in the family $K.$

We have
$$ (3,4, 2,1,1) = (4,3,2,1,1)+ (3,3,2,2,1)+ (3,3,2,1,2)$$
as $(3,4, 2,1,1) =(4,3,2,1,1)+ (3,3,2,2,1)+ (3,3,2,1,2) +\Sq^1(3,3,2,1,1). $

Similarly, 
\begin{align*}
(3,4, 1,2,1) &  =(4,3,1,2,1)+ (3,3,2,2,1)+ (3,3,1,2,2) \\
(3,4, 1,1,2) & =(4,3,1,1,2)+ (3,3,1,2,2)+ (3,3,2,1,2).
\end{align*}
We also have 
$$(4,3,1,1,2) = (4,3,2,1,1)+ (4,3,1,2,1),$$
because $(4,3,1,1,2) = (4,3,2,1,1)+ (4,3,1,2,1) 
+(4,4,1,1,1)+\Sq^1(4,3,1,1,1).$

Let $h_1 $ and $h_2$ be the coefficients respectively of the monomials 
$(4,3,2,1,1)$  and $(4,3,1,2,1)$  in an expression of $p$ . Then
$$p = h_1(4,3,2,1,1)+h_2(4,3,1,2,1) + \text{ other terms}.$$
On the other hand, $p$ is a $GL_5$--invariant, so
$$\omega_{xy}(p) + p = 0. $$
We have 
\begin{align*}
\omega_{xy}(4,3,2,1,1) & = (4,3,2,1,1) +  (0,7,2,1,1) \\
\omega_{xy}(4,3,1,2,1) & = (4,3,1,2,1)  + (0,7,1,2,1), 
\end{align*}
and the images under $\omega_{xy}$ of any other monomials in the 
expression of $p$ 
do not contain the monomials $(0,7,2,1,1), (0,7,1,2,1), (0,7,1,1,2)$ as 
terms.  Thus, 
$$\omega_{xy}(p) + p = h_1(0,7,2,1,1)+h_2(0,7,1,2,1) + 
\text{ other terms not in }C.$$
So, we get 
$$h_1(0,7,2,1,1)+h_2(0,7,1,2,1) = 0.$$
Applying $\omega_{ut}$, which sends $t$ to $t+u$ and keeps the other 
variables fixed, to this equality, we obtain
$$ h_1(0,7,2,2,0)+ h_2(0,7,1,3,0) = 0.$$
This implies $h_2 =0$, as $(0,7,1,3,0)$ is a spike.   Similarly, we have 
$h_1 = 0.$

We have shown that in the expression of $p$, the sum  of 
monomials of the forms $(4,3,c,d,e)$ and $(3,4,c,d,e)$ in $H$  
can be written in terms of monomials in the family $K.$

Because of the symmetry of the variables, the above argument also works 
for the sum  of monomials in $H$  in the expression of $p$ 

\textbf{Step 2}\qua We will show that if $p \in \calL(K) $  is a '
$GL_5$--invariant, then $p$ equals zero.

Note that if $p \in \calL(K),$ then it  is expressed in the terms of 
the variables permutations of the monomial $(3,3,2,2,1).$
Let $k_{(a,b,c,d,e)}$ be the coefficient of the monomial $(a,b,c,d,e)$ 
in an expression of $p$.  
Because of the symmetry of the variables, in order to prove $p = 0$ 
we  need only to prove $k_{(2,2,3,3,1)} = 0.$

There are exactly three monomials of the form $(a,b,c,3,1)$ in $K$, namely
$$ (3,2,2,3,1), (2,3,2,3,1), (2,2,3,3,1).$$ 
Let $\sigma$ be the transformation defined in the proof of \fullref{E}, 
which  sends $x$ to $x+z$,  $y$ to $y+z$ and fixes the other variables.  

A routine computation shows
\begin{eqnarray*}
\sigma (3,2,2,3,1) &= & (3,2,2,3,1) + (3,0,4,3,1)+(2,2,3,3,1)+(2,0,5,3,1) 
\\
&   & +(1,2,4,3,1) +(1,0,6,3,1)+(0,2,5,3,1)+(0,0,7,3,1) \\
\sigma (2,3,2,3,1) &= & (2,3,2,3,1) +(0,3,4,3,1)+(2,2,3,3,1)+(0,2,5,3,1) \\
&    & + (2,1,4,3,1)+(0,1,6,3,1)+(2,0,5,3,1)+(0,0,7,3,1)\\
\sigma (2,2,3,3,1) &= & (2,2,3,3,1)+(0, 2, 5, 3, 
1)+(2,0,5,3,1)+(0,0,7,3,1).
\end{eqnarray*}
Further, the images under $\sigma$ of the other terms in the expression 
of $p$ do not contain $(0,0,7,3,1)$ as a term,  because the exponents 
of $t$ and $u$ in these monomials are not respectively $3$ and $1$.  
So, $\sigma (p)+p$  contains $(k_{(3,2,2,3,1)}+ 
k_{(2,3,2,3,1)}+k_{(2,2,3,3,1)}) (0,0,7,3,1)$ as a term.
Moreover, as $p$ is a $GL_5$--invariant, $\sigma (p) + p = 0$. 
It implies  
$$k_{(3,2,2,3,1)}+ k_{(2,3,2,3,1)}+k_{(2,2,3,3,1)} = 0.$$
On the other hand, consider the set of monomials of the form 
$(a,b, 2, d, 1)$ 
and\break $(a,b,1,d,2)$ in the family $K$.  Then, $(a,b,d)$ is a permutation 
of $(3,3,2).$
We have
\begin{eqnarray*}
\omega_{xy}(3,3,2,2,1) +(3,3,2,2,1) & = & (2,4,2,2,1) +(1,5,2,2,1) 
+(0,6,2,2,1) \\
\omega_{xy}(3,2,2,3,1) +(3,2,2,3,1) & = & (2,3,2,3,1) +(1,4,2,3,1) 
+(0,5,2,3,1) \\
\omega_{xy}(2,3,2,3,1) +(2,3,2,3,1) & = & (0,5,2,3,1).
\end{eqnarray*}
Let $\omega_{yt}$ be the transformation that sends $y$ to $y+t$ and keeps 
the other variables fixed. Apply $\omega_{yt}$ to $\omega_{xy}(p) + p,$ we 
have
$$ \omega_{yt} (0,5,2,3,1) =  (0,5,2,3,1) + (0,4,2,4,1)+ 
(0,1,2,7,1)+(0,0,2,8,1) .$$
It is easy to see that the actions of $\omega_{xy}$ and $\omega_{yt}$ on 
the monomial do  not change the exponents of $z$ and $u$. Combining this 
with the fact that the exponents of $z$ and $u$ in the other monomials  
are not respectively  $2$ and $1,$ it implies $\omega_{yt}(\omega_{xy}(p) 
+p) $ contains  $(k_{(3,2,2,3,1)}+ k_{(2,3,2,3,1)})(0,1,2,7,1)$ as a term.

Similarly, $\omega_{yt}(\omega_{xy}(p) +p)$ contains $ (k_{(3,2,1,3,2)}+ 
k_{(2,3,1,3,2)})(0,1,1,7,2)$ as a term.

Further, both the exponents of $z$ and $u$ in the other monomials are not 
equal to 1. So, their image under the action $\omega_{yt}, \omega_{xy}$ 
does not contain the monomial $(0,2,1,7,1).$

Hence
\begin{eqnarray*}
\omega_{yt}(\omega_{xy}(p) +p) &  = &(k_{(3,2,2,3,1)}+ 
k_{(2,3,2,3,1)})(0,1,2,7,1)\\
& &+(k_{(3,2,1,3,2)}+ k_{(2,3,1,3,2)})(0,1,1,7,2) \\
&&+\text{ other term  is not in }
\pi_x(C).
\end{eqnarray*}
As $\omega_{xy}(p) +p = 0,$ we have
$$(k_{(3,2,2,3,1)}+ k_{(2,3,2,3,1)})(0,1,2,7,1)+(k_{(3,2,1,3,2)}+ 
k_{(2,3,1,3,2)})(0,1,1,7,2) = 0.$$
As shown in the proof of \fullref{CD}, this implies 
$$ k_{(3,2,2,3,1)}+ k_{(2,3,2,3,1)} = 0.$$
As a consequence
$$k_{(2,2,3,3,1)} = 0. \proved$$
\end{proof}


\section{The fifth  algebraic transfer is not an epimorphism}
\label{sec4}

The target of this section is to prove the following theorem, which is 
also 
numbered as \fullref{main} in the introduction. 

\begin{Theorem}\label{mainagain}
The element $P(h_2) \in \Ext_{\cala}^{5, 16}(\Fd,\Fd)$ is not in the 
image of the 
algebraic transfer $\Tr_5 \co \Fd\otiglfive PH_{11}(B\V_5) \to 
\Ext_{\cala}^{5,16}(\Fd,\Fd)
$. 
\end{Theorem}

\begin{proof}          
According to \fullref{invariant}, we have
$$
(\Fd\oticala P_5)^{GL_5}_{11} = 0.
$$
As $\Fd\otiglfive PH_*(B\V_5)$ is dual to $(\Fd\oticala P_5)^{GL_5}$, we 
get
$$
\Fd\otiglfive PH_{11}(B\V_5) = 0. 
$$
It is well known (see, for example, M\,C Tangora \cite{Tangora} and
R\,R Bruner \cite{Bruner}) that the 
element  
$P(h_2)$ is nonzero in  $\Ext_{\cala}^{5,16}(\Fd,\Fd)$. 
So, the fifth algebraic  transfer
$$
\Tr_5 \co \Fd\otiglfive PH_{11}(B\V_5) \to \Ext_{\cala}^{5,16}(\Fd,\Fd)
$$
does not detect the nonzero element $P(h_2).$
\end{proof}

\nocite{Hung97}

\bibliographystyle{gtart}
\bibliography{link}

\begin{thebibliography}{}
\providecommand\bibmarginpar{\leavevmode\marginpar}
\def\urlstyle#1{{\tt #1}}

\bibitem{Boardman}
\textbf{J\,M Boardman}, \emph{Modular representations on the homology of powers
  of real projective space}, from: ``Algebraic topology (Oaxtepec, 1991)'',
  Contemp. Math. 146, Amer. Math. Soc., Providence, RI (1993)  49--70
  \xox{MR}{1224907}

\bibitem{Bruner}
\textbf{R\,R Bruner}, \emph{The cohomology of the mod 2 Steenrod algebra: A
  computer calculation}, WSU Research Report 37 (1997)

\bibitem{BHH}
\textbf{R\,R Bruner}, \textbf{L\,M H{\`a}}, \textbf{N\,H\,V H\uhorn{}ng},
  \href{http://dx.doi.org/10.1090/S0002-9947-04-03661-X} {\emph{On the behavior
  of the algebraic transfer}}, Trans. Amer. Math. Soc. 357 (2005) 473--487
  \xox{MR}{2095619}

\bibitem{Hung97}
\textbf{N\,H\,V H\uhorn{}ng},
  \href{http://dx.doi.org/10.1090/S0002-9947-97-01991-0} {\emph{Spherical
  classes and the algebraic transfer}}, Trans. Amer. Math. Soc. 349 (1997)
  3893--3910 \xox{MR}{1433119}

\bibitem{Hung}
\textbf{N\,H\,V H\uhorn{}ng},
  \href{http://dx.doi.org/10.1090/S0002-9947-05-03889-4} {\emph{The cohomology
  of the {S}teenrod algebra and representations of the general linear groups}},
  Trans. Amer. Math. Soc. 357 (2005) 4065--4089 \xox{MR}{2159700}

\bibitem{Singer}
\textbf{W\,M Singer}, \href{http://dx.doi.org/10.1007/BF01221587} {\emph{The
  transfer in homological algebra}}, Math. Z. 202 (1989) 493--523
  \xox{MR}{1022818}

\bibitem{Singer91}
\textbf{W\,M Singer},
  \href{http://links.jstor.org/sici?sici=0002-9939(199102)111:2%3C577:OTAOSS%3%
E2.0.CO%3B2-K} {\emph{On the action of {S}teenrod squares on polynomial
  algebras}}, Proc. Amer. Math. Soc. 111 (1991) 577--583 \xox{MR}{1045150}

\bibitem{Tangora}
\textbf{M\,C Tangora}, \href{http://dx.doi.org/10.1007/BF01110185} {\emph{On
  the cohomology of the {S}teenrod algebra}}, Math. Z. 116 (1970) 18--64
  \xox{MR}{0266205}

\end{thebibliography}

\end{document}